\documentclass{amsart}
\usepackage{amsmath,amssymb,amsthm,amsfonts,mathrsfs}
\usepackage[mathscr]{eucal}
\usepackage{enumerate}
\usepackage{mathtools}

\usepackage{tikz-cd}

\usepackage[pdftex,colorlinks,backref=page,citecolor=blue]{hyperref}

\usepackage[paperheight=11in,
            paperwidth=8.5in,
            margin=1in]{geometry}
\usepackage[T1]{fontenc}
\usepackage[utf8]{inputenc}
\usepackage[default,regular,semibold]{sourceserifpro} 
\usepackage{enumitem}
\setlist[description]{font=\rmfamily\mdseries,labelindent=\parindent}

\usepackage[capitalize]{cleveref}
\crefname{equation}{}{}
\crefname{enumi}{}{}
\creflabelformat{enumi}{(#2#1#3)}

\theoremstyle{definition}
\newtheorem{dfn}{Definition}[section]
\newtheorem{thm}[dfn]{Theorem}
\newtheorem{prop}[dfn]{Proposition}
\newtheorem{lem}[dfn]{Lemma}
\newtheorem{cor}[dfn]{Corollary}

\newtheorem{remark}[dfn]{Remark}
\newtheorem*{dfn*}{Definition}
\newtheorem{manualbackend}[dfn]{\!}
\newenvironment{manual}[1]{
  \renewcommand{\thedfn}{#1}
  \addtocounter{dfn}{-1}
  \begin{manualbackend}
  }{
  \end{manualbackend}
}
\theoremstyle{remark}
\newtheorem{zbbackend}[dfn]{Example}
\newenvironment{example}[1][]{\begin{zbbackend}[#1]}{\hspace*{\fill}$\diamondsuit$\end{zbbackend}}
\newtheorem{rmk}[dfn]{Remark}

\DeclareMathOperator{\DP}{DP} 
\DeclareMathOperator{\rk}{rk} 
\DeclareMathOperator{\cl}{cl} 
\DeclareMathOperator{\codim}{codim}
\DeclareMathOperator{\cone}{cone}
\DeclareMathOperator{\im}{im}
\DeclareMathOperator{\Lowest}{Lowest}
\newcommand{\geo}{{\hspace{0.05ex}\mathrm{geo}}}
\newcommand{\ph}{\varphi}
\newcommand{\Q}{\mathbb{Q}} 
\newcommand{\R}{\mathbb{R}} 
 
\newcommand{\Z}{\mathbb{Z}} 
\renewcommand{\L}{\mathscr{L}} 
\newcommand{\G}{\mathscr{G}} 
 
\newcommand{\N}{\mathscr{N}} 
\newcommand{\cS}{\mathscr{S}} 
\newcommand{\sym}{\mathfrak{S}}

\renewcommand{\P}{\mathbb{P}} 

\newcommand{\e}{\mathbf{e}}

\newcommand{\word}[1]{\textit{#1}}

\title{The Bergman fan of a polymatroid}
\author{Colin Crowley}
\address{
University of Oregon
}
\email{crowley@uoregon.edu}

\author{June Huh}
\address{
Princeton University and Korea Institute for Advanced Study
}
\email{huh@princeton.edu}

\author{Matt Larson}
\address{
Stanford University
}
\email{mwlarson@stanford.edu}

\author{Connor Simpson}
\address{
University of Wisconsin–Madison
}
\email{csimpson6@wisc.edu}

\author{Botong Wang}
\address{
University of Wisconsin–Madison
}
\email{wang@math.wisc.edu}

\begin{document}

\begin{abstract}
    We introduce the Bergman fan of a polymatroid and prove that the Chow ring of the Bergman fan is isomorphic to the Chow ring of the polymatroid. Using the Bergman fan, we establish the K\"ahler package for the Chow ring of the polymatroid, recovering and strengthening a result of  Pagaria--Pezzoli. 
\end{abstract}

\maketitle

\section{Introduction}\label{intro}

By definition, a \emph{matroid} on a finite set $E$ is given by a rank function
$ \rk\colon  2^E \to \Z_{\geq 0} $
satisfying the following:
\begin{enumerate}[]\itemsep 5pt
\item[] (Submodularity) For any $A_1,A_2 \subseteq E$, we have $\rk(A_1 \cup A_2) + \rk(A_1 \cap A_2) \leq \rk(A_1) + \rk(A_2)$.
\item[] (Monotonicity) For any $A_1 \subseteq A_2 \subseteq E$, we have $\rk(A_1) \leq \rk(A_2)$.
\item[] (Boundedness) For any $A \subseteq E$, we have $\rk(A) \leq |A|$.
\item[] (Normalization) The rank of the empty subset is zero.
\end{enumerate}
Sans ``boundedness'', the axioms above define a \textit{polymatroid}.
Throughout this paper, we assume that the polymatroid is \emph{loopless}:
\begin{enumerate}[]\itemsep 5pt
\item[] (Looplessness) The rank of any nonempty subset is nonzero.
\end{enumerate}
If $P$ is a polymatroid on $E$, then its \emph{rank} is $\rk(P) \coloneqq \rk(E)$.
A \textit{flat} of $P$ is a subset $F \subseteq E$ that is maximal among sets of its rank.
Ordered by inclusion, the flats of $P$ form a lattice $\L_P$.\footnote{Unlike in the case of matroids, the lattice of flats of a polymatroid can fail to be graded or atomic.} 
The intersection of two flats is a flat, so any subset $A$ of $E$ is contained in a unique minimal flat $\cl_P(A)$, called the \emph{closure} of $A$ in $P$, which is obtained by intersecting all flats that contain $A$.

Matroids can be viewed as combinatorial abstractions of hyperplane arrangements.
More generally, polymatroids can be viewed as combinatorial abstractions of subspace arrangements.

\begin{example}\label{subspace}
  Let $V_1, \ldots, V_n$ be linear subspaces of a vector space $V$ over a field $\mathbb{F}$.
  The rank function
  \[ \rk(A) \coloneqq \codim_V(\cap_{i \in A} V_i) \]
   defines a polymatroid $P$ on the set of indices $E=\{1,\ldots,n\}$, which is a matroid if and only if every $V_i$ is a hyperplane.
   The map
   $F \mapsto \cap_{i \in F} V_i$
   is a bijection between the flats of $P$ and the subspaces of $V$ obtained by intersecting some of the $V_i$'s.
   A polymatroid arising in this way is said to be \word{realizable} over $\mathbb{F}$, and the subspace arrangement is called a \emph{realization} of $P$ over $\mathbb{F}$. 
\end{example}

Much of a hyperplane arrangement's combinatorial data is captured by intersection theory on its wonderful compactification \cite{dCP95}.
To extend this from hyperplane arrangements to non-realizable matroids, one must replace the wonderful compactification with a combinatorial object, the Bergman fan of a matroid.
The purpose of the present paper is to introduce the \emph{Bergman fan of a polymatroid}, a combinatorial model for the wonderful compactification of a subspace arrangement.
As in the case of matroids, the Bergman fan of a polymatroid is a tropical variety of degree one.
In Section \ref{sec:chow}, we show that the Chow ring of a polymatroid satisfies the K\"{a}hler package with respect to any strictly convex piecewise linear function on its Bergman fan, 
 recovering and strengthening a result of Pagaria and Pezzoli \cite[Theorems 4.7 and 4.21]{PP21}.

Our construction of the Bergman fan is inspired by a geometric observation: over an infinite field, the wonderful compactification of any subspace arrangement can be realized as the wonderful compactification of a hyperplane arrangement, taken with respect to an appropriate building set (\cref{rmk:geomOfBS}).
The construction immediately reveals that the Bergman fan of a polymatroid and the Bergman fan of the associated matroid have the same support.
Thus, the K\"{a}hler package for the polymatroid follows from that of the associated matroid \cite{AHK18}
and the general fact that the validity of the K\"{a}hler package for the Chow ring of a fan depends only on the support of the fan \cite{ADH20}.

\subsection{The Bergman fan of a Boolean polymatroid}\label{intro:boolean}
An important special case is that of Boolean polymatroids.
Let $\pi\colon \widetilde E \to E$ be a surjective map between finite sets.
The \textit{Boolean polymatroid} $B(\pi)$ is the polymatroid on $E$ defined by the rank function 
\[
\rk_{B(\pi)}(A) = |\pi^{-1}(A)| \ \ \text{for $A \subseteq E$.}
\]
We write $N_{\widetilde E}$ for $\mathbb{Z}^{\widetilde E} / \mathbb{Z}(1,1,\ldots,1)$,
and, for a subset $S$ of $\widetilde E$, write $\e_S$ for the vector
$\sum_{i \in S} \e_i$ in $N_{\widetilde E} \otimes \R$.

\begin{dfn}\label{dfn:booleanBergman}
 The \word{Bergman fan} $\Sigma_{B(\pi)}$ of the Boolean polymatroid $B(\pi)$ is the fan in $N_{\widetilde E}\otimes \R$ with cones
  \[
    \sigma_{\mathscr F, A} \coloneqq \cone( \e_{\pi^{-1}(F_1)}, \ldots, \e_{\pi^{-1}(F_k)}) + \cone(\e_i)_{i \in S},
  \]
  for every chain $\mathscr{F} = \{\emptyset \subsetneq F_1 \subsetneq F_2 \subsetneq \cdots \subsetneq F_k \subsetneq E \}$
  and subset $S$ of $\widetilde E$ not containing any fiber of $\pi$.
\end{dfn}

Throughout the paper, we write $n$ for the cardinality of $E$.

\begin{dfn}
  An \emph{ordered transversal} of $\pi$ is a sequence $s_1,\ldots,s_n$ of elements of $\tilde{E}$ such that each fiber of $\pi$ contains exactly one element of the sequence.
  The \textit{polypermutohedron}  $Q(\pi)$ is the convex hull of the vectors
  $\sum_{i =1}^n i \e_{s_i}$ in $\R^{\widetilde E}$, 
    where $s_1, \ldots, s_n$ range over all ordered transversals of $\pi$.
\end{dfn}

In \cref{appendix:polytope}, we show $\Sigma_{B(\pi)}$ is the inner normal fan of the polypermutohedron $Q(\pi)$. It follows immediately from Definition \ref{dfn:booleanBergman} that $\Sigma_{B(\pi)}$ is a complete unimodular fan in $N_{\widetilde E}$.

\begin{example}
When $\pi$ is a bijection, an ordered traversal of $\pi$ is a permutation of $\tilde{E}$, and $Q(\pi)$ is the standard permutohedron in $\R^{\widetilde E}$. This recovers a familiar statement: the Bergman fan of the Boolean matroid is the normal fan of the standard permutohedron.
\end{example}
\begin{example}
When $E$ is a singleton, an ordered transversal of $\pi$ is an element $\tilde{E}$, and $Q(\pi)$ is the standard simplex in $\R^{\widetilde E}$.
Thus, the Bergman fan of a Boolean polymatroid on a singleton is the normal fan of the standard simplex.
\end{example}

\begin{example}
When $|E| = n-1$ and all fibers of $\pi$ have size $d$, the toric variety corresponding to $\Sigma_{B(\pi)}$ is a  generalization of the Losev--Manin space of curves which compactifies the moduli space of configurations of $n$ points in $\mathbb{A}^d$ up to translation and scaling \cite[Corollary 5.6]{GallardoRoutis}. 
\end{example}

\subsection{The Bergman fan of a polymatroid}

Let $P$ be a polymatroid on $E$, and let $\pi\colon \widetilde E \to E$ be a surjective map satisfying $\rk_P(i)=|\pi^{-1}(i)|$ for every $i$ in $E$.

\begin{dfn}\label{dfn:maximalBergman}
   The \word{Bergman fan} $\Sigma_P$ of the polymatroid $P$ is the subfan of $\Sigma_{B(\pi)}$ with cones
  \[
    \sigma_{\mathscr F, S} \coloneqq \cone( \e_{\pi^{-1}(F_1)}, \ldots, \e_{\pi^{-1}(F_k)}) + \cone(\e_i)_{i \in S},
  \]
  one for every chain of flats $\mathscr F = \{\emptyset = F_0 \subsetneq F_1 \subsetneq F_2 \subsetneq \cdots \subsetneq F_k \subsetneq E \}$ of $P$ and a subset $S$ of $\widetilde E$ such that 
\[
\rk(F \cup \pi(T)) > \rk(F) + |T| \ \ \text{for every proper flat $F$ in $\mathscr{F}$ and every nonempty subset $T$ of $S\setminus \pi^{-1}(F)$.}
\]
\end{dfn}

The Bergman fan $\Sigma_P$ is unimodular with respect to $N_{\widetilde E}$, defining a smooth toric variety $X_P$ over $\mathbb{C}$. We write $A(\Sigma_P)$ for the Chow ring of $X_P$. 
We relate $A(\Sigma_P)$ to the \emph{Chow ring} of the polymatroid $P$ (Definition \ref{dfn:ppchow}), denoted $\DP(P)$, introduced in \cite[Section 4]{PP21}. Our main result states the following.
 
\begin{manual}{\cref{thm:isomorphism}}
 There is a natural isomorphism of graded rings $\DP(P) \cong A(\Sigma_P)$.
\end{manual}

In \cref{cor:gb}, we use \cref{thm:isomorphism} to recover a Gr\"obner basis for $\DP(P)$ found in \cite{PP21}. 
In \cref{cor:hodge}, we prove the K\"{a}hler package for $A(\Sigma_P)$ with respect to the cone of strictly convex piecewise linear function on $\Sigma_P$, extending the K\"{a}hler package for $\DP(P)$ with respect to the $\sigma$-cone in  \cite{PP21}. See Remark \ref{ConeRemark} for a comparison of the two cones.

\subsection{Building sets}
In fact, our results hold for polymatroids $P$ equipped with a \word{geometric building set} $\G$ (\cref{sec:buildingsets}). The statements above are specializations of our results to the case when $\G$ consists of all nonempty flats of $P$.
In maximal generality, we define the \word{Bergman fan} of $(P, \G)$, denoted $\Sigma_{P,\G}$ (\cref{dfn:bergmanfan}).
The Chow ring $\DP(P, \G)$ associated to $(P, \G)$ was introduced in \cite{PP21}, and it is isomorphic to 
$A(\Sigma_{P, \G})$ (\cref{thm:isomorphism}).
All corollaries continue to hold, including the K\"{a}hler package for $\DP(P, \G)$.

\subsection*{Organization}
In \cref{sec:multisymmetric}, we develop the combinatorics of multisymmetric matroids and lifts, a key tool throughout this paper.
We use lifts to define the Bergman fan of a polymatroid (with respect to a geometric building set) in \cref{sec:buildingsets}.
Finally, in \cref{sec:chow}, we show that the Chow ring of the Bergman fan agrees with the polymatroid Chow ring of \cite{PP21} and derive consequences.
Examples \ref{subspace} and \ref{rmk:why-geometric}, and Remarks \ref{rmk:geomOfBS}, \ref{tropicalization}, and \ref{rmk:geoProof} explain the geometry underlying this work.

\subsection*{Acknowledgements}
We thank Spencer Backman for comments and conversations. The second author is partially supported by a Simons Investigator Grant and NSF Grant DMS-2053308, the third is supported by an NDSEG fellowship, and the last is supported by a Sloan fellowship.

\section{Multisymmetric matroids}\label{sec:multisymmetric}
Many proofs in \cref{sec:chow} reduce statements about polymatroids to known statements about matroids.
The key tool for this reduction is multisymmetric matroids, a new cryptomorphic formulation of polymatroids. 
Let $\sym_{\widetilde{E}}$ denote the symmetric group on a finite set $\widetilde{E}$. We continue to assume that all  (poly)matroids are loopless.
\begin{dfn}\label{dfn:multisymmetric}
  A \emph{multisymmetric matroid} is a matroid $M$ on $\widetilde{E}$ equipped with a partition
  $\widetilde{E} = \widetilde{E}_1 \sqcup \cdots \sqcup \widetilde{E}_n$
  such that the action of $\Gamma = \sym_{\widetilde{E}_1} \times \cdots \times \sym_{\widetilde{E}_n}$ on $\widetilde{E}$ takes flats to flats.
The \word{geometric part} of a subset $S \subseteq \widetilde{E}$ is $S^\geo \coloneqq \cap_{\gamma \in \Gamma} (\gamma \cdot S)$. We call a subset $S\subseteq \widetilde{E}$ \word{geometric} if $S = S^\geo$.
\end{dfn}

For a multisymmetric matroid $M$, we write $\L_M^\Gamma$ for its poset of geometric flats.

\begin{example}
  Any matroid $M$ on $\widetilde{E}$ can be given the \word{trivial} multisymmetric structure by setting $\Gamma = \prod_{e \in \widetilde{E}} \sym_{\{e\}}$. In this case, $\L_M^\Gamma = \L_M$.
\end{example}
\begin{example}\label{ex:uniform}
  If $M$ is multisymmetric on $\widetilde{E} = \widetilde{E}_1$, then $M$ is a uniform matroid, and $\L_M^\Gamma$ is $\{\emptyset, \widetilde{E}\}$.
\end{example}
\begin{example}
  Let $\widetilde{E}$ be the set of edges of the complete graph $K_4$, and let $M$ be the graphic matroid of $K_4$.
  No transposition of $\sym_{\widetilde{E}}$ preserves the flats of $M$, so $M$ has no non-trivial multisymmetric structures.
\end{example}

Closure in a multisymmetric matroid is restricted by the group action.
\begin{lem}\label{cor:union}
  If $M$ is multisymmetric and $S \subseteq \widetilde{E}$ is geometric, then $\cl_M(S)$ is also geometric.
\end{lem}
\begin{proof}
For any $S \subseteq \widetilde{E}$ geometric and $\gamma \in \Gamma$, we have $\gamma \cdot \cl_M(S) = \cl_M(\gamma\cdot S) = \cl_M(S)$.
\end{proof}

\begin{cor}\label{cor:sublattice}
  If $M$ is multisymmetric, then the geometric flats of $M$ form a sublattice of $\L_M$.
\end{cor}
\begin{proof}
  If $F$ and $G$ are two geometric flats, then $\cl_M(F \cup G)$ is geometric by \cref{cor:union}.
  The intersection of two geometric flats is also geometric.
  In other words, the set of geometric flats is closed under both join and meet, and therefore forms a sublattice of $\L_M$.
\end{proof}
\begin{cor}\label{cor:polymat}
  If $M$ is multisymmetric on $\widetilde{E} = \widetilde{E}_1 \sqcup \cdots \sqcup \widetilde{E}_n$, then $\L_M^\Gamma$ is the lattice of flats of the polymatroid $P$ on the set of indices $E = \{1, \ldots, n\}$ defined by the rank function
    $\rk_P(A) \coloneqq \rk_M( \cup_{i \in A} \widetilde{E}_i).$ 

\end{cor}
\begin{proof}
  Define $\pi\colon  \widetilde{E} \to E$ by setting $\pi^{-1}(i) = \widetilde{E}_i$.
If $F$ is a flat of $P$, then for all $F \subsetneq A \subseteq \{1, \dotsc, n\}$,
\[ \rk_M(\pi^{-1}(F)) = \rk_P(F) < \rk_P(A) = \rk_M(\pi^{-1}(A)). \]
By \cref{cor:union}, $\cl_M(\pi^{-1}(F))$ is geometric, so we conclude that $\cl_M(\pi^{-1}(F)) = \pi^{-1}(F)$. 
In other words, $\pi^{-1}(F)$ is a flat of $M$.
Conversely, if $F$ is not a flat of $P$, then $\rk_P(F) = \rk_P(F \cup i)$ for some $i$ not in $F$.
This implies $\cl_M(\pi^{-1}(F)) \supseteq \pi^{-1}(F \cup i)$, so $\pi^{-1}(F)$ is not a flat of $M$.
Therefore, $F$ is a flat of $P$ if and only if $\pi^{-1}(F)$ is a geometric flat of $M$, and we have an isomorphism of lattices
\[ \L_P \longrightarrow \L_M^\Gamma, \qquad F \longmapsto \pi^{-1}(F). \qedhere\]
\end{proof}

\begin{lem}\label{lem:closure}
  Let $M$ be a multisymmetric matroid on $\widetilde{E} = \widetilde{E}_1 \sqcup \cdots \sqcup \widetilde{E}_n$. 
\begin{enumerate}\itemsep 5pt
  \item If $S \subseteq \widetilde{E}$, then either $\cl_M(S) \cap \widetilde{E}_i = \widetilde{E}_i$ or $\cl_M(S) \cap \widetilde{E}_i = S \cap \widetilde{E}_i$.\label{lem:extraclosure}
  \item If $F$ is a flat of $M$, then $\rk_M(F) = \rk_M(F^\geo) + |F \setminus F^\geo|$. \label{lem:closurerank}
\end{enumerate}
\end{lem}

\begin{proof}
    A permutation of $\widetilde{E}_i \setminus S$ induces an automorphism of $M$ that fixes $S$. Any such automorphism also fixes $\cl_M(S)$ because automorphisms commute with closure.
    Hence, if $(\widetilde{E}_i \setminus S) \cap \cl_M(S)$ is nonempty, then $\widetilde{E}_i \subseteq \cl_M(S)$. This proves the first part.
     
      For the second part, let $\{s_1, \ldots, s_k\} = F \setminus F^\geo$. If \cref{lem:closurerank} fails then for some $1 \leq i \leq k,$
       \[
         \rk_M(F^\geo\cup \{s_1, \ldots, s_i\}) = \rk_M(F^\geo\cup \{s_1, \ldots, s_{i+1}\}).
       \]
       Therefore if $s_{i+1} \in \widetilde{E}_j$ then $\widetilde{E}_j \subseteq \cl_M(F^\geo\cup \{s_1, \ldots, s_i\}) \subseteq F$ by the first part.
        Consequently, $s_{i+1} \in F^\geo$, a contradiction. 
\end{proof}

\begin{lem}\label{uniqueness}
    A multisymmetric matroid is determined by its geometric sets and their ranks.
\end{lem}
\begin{proof}
    Let $M$ and $M'$ be two multisymmetric matroids on $\widetilde E = \widetilde E_1 \sqcup \cdots \sqcup \widetilde E_n$, and suppose that for all $A \subset \{1, \ldots, n\}$, $\rk_M(\cup_{i \in A} \widetilde E_i) = \rk_{M'}(\cup_{i \in A} \widetilde E_i)$.
   If $F$ is a flat of $M$, and  $F'$ is the closure of $F$ in $M'$, then
   \[
     \rk_{M'}(F) \leq \rk_{M'}(F^\geo) + \rk_{M'}(F \setminus F^{\geo})
     \leq \rk_{M}(F^\geo) + |F \setminus F^{\geo}| = \rk_M(F)
   \]
   by \cref{lem:closure}\cref{lem:closurerank}.
   Symmetrically, $\rk_M(F') \leq \rk_{M'}(F')$, so
   \[
     \rk_{M'}(F) \leq \rk_M(F) \leq \rk_{M}(F') \leq \rk_{M'}(F').
   \]
   The left- and rightmost terms are equal, so $F = F'$ because $F$ is a flat of $M$. This shows that $M$ and $M'$ have the same flats, and that their flats have the same ranks, so $M$ and $M'$ are equal.    
 \end{proof}
 \begin{rmk}
   The closure of a geometric set is a geometric flat by \cref{lem:closure}\cref{lem:extraclosure}, so \cref{uniqueness} implies that a multisymmetric matroid is also determined by its geometric flats and their ranks.
 \end{rmk}

\subsection{Lifts}
  Let $P$ be a polymatroid on $E = \{1, \ldots, n\}$ and
   $M$ a multisymmetric matroid on $\widetilde{E}~=~\widetilde{E}_1~\sqcup~\cdots~\sqcup~\widetilde{E_n}$. If $P$ is the polymatroid given by Corollary \ref{cor:polymat}, then we say that $M$ is a \word{multisymmetric lift} of $P$.
  If $\rk_M(\widetilde{E}_i) = |\widetilde{E_i}|$ for all $1 \leq i \leq n$, then we say that $M$ is a \word{minimal} multisymmetric lift of $P$.
\begin{thm}\label{lift}
A polymatroid $P$ has a unique minimal multisymmetric lift $\widetilde P$ constructed as follows.
  For $1 \leq i \leq n$, let $\widetilde E_i = \{1, \ldots, \rk_P(i)\}$, and 
$ \widetilde E = \widetilde E_1 \sqcup \cdots \sqcup \widetilde E_n$. Define the projection $\pi\colon  \widetilde E \to E$ by $\pi^{-1}(i) = \widetilde E_i$.

  The minimal multisymmetric lift of $P$ is the matroid $\widetilde P$ on $\widetilde E$ with rank function
  \[
    \rk_{\widetilde P}(S) = \min \{ \rk_P(A) + |S \setminus \pi^{-1}(A)| : A \subseteq E \}.
\]
\end{thm}
\begin{example}\label{rmk:why-geometric}
  Suppose $P$ is the polymatroid realized by an arrangement of subspaces $V_1,\ldots, V_n$ in $V$ as in \cref{subspace}.
  The minimal multisymmetric lift $\widetilde P$ is realized by any hyperplane arrangement
  \[ \{V_{i,j} : 1 \leq i \leq n, \; 1 \leq j \leq \rk_P(i)\}, \]
  where $V_{i,1}, \ldots, V_{i, \rk_P(i)}$ are generic hyperplanes containing $V_i$.
  Under the correspondence between flats of $\widetilde P$ and intersections of the $V_{i,j}$'s, the geometric flats correspond to the subspaces of $V$ arise in for any choice of $\{V_{i,j}\}_{i,j}$.
\end{example}
  Versions of the construction in \cref{lift} make many independent appearances in the literature, for example, in \cite[\S2]{H72} and \cite[Propositions 3.1 and 3.2]{L77}. The most complete treatment we are aware of is \cite[\S2]{N86}, whose terminology differs from ours.
\subsubsection*{Notation.} We continue to use the notations of \cref{lift} in the remainder of this section. For visual clarity, we often write $M = \widetilde P$.
As usual, $\Gamma$ stands for the product of symmetric groups that acts on $M$.

\medskip
\begin{prop}\label{liftwelldefined}
The minimal lift $M = \widetilde P$ is a multisymmetric matroid. Explicitly,
\begin{enumerate}\itemsep 5pt
\item $\rk_{M}$ is a matroid rank function, and
\item the action of $\Gamma$  preserves flats.
\end{enumerate}
\end{prop}
\begin{proof}
 The proof of the first part is reproduced from \cite[Proof of Theorem 11.1.9]{O11}. We need to check that $\rk_M$ is non-negative, increasing, submodular, and satisfies $\rk_M(S)  \leq |S|$ for $S \subseteq \widetilde E$. 

It is clear that values of $\rk_M$ are non-negative.
Let $S \subseteq \widetilde E$ and $s \in \widetilde E \setminus S$. For all $A \subseteq E$,
\[
\rk_P(A) + |S \setminus \pi^{-1}( A)| \leq \rk_P(A) + |S\cup a \setminus \pi^{-1}( A)|,
\]
so $\rk_M$ is increasing. Moreover, by induction on $|S|$,
\begin{align*}
  \rk_M(S \cup s)
  =& \min\{\rk_P(A) + |S \cup s \setminus \pi^{-1}(A)| : A \subseteq E\}\\
  \leq& \min\{\rk_P(A) + |S \setminus \pi^{-1}(A)| : A \subseteq E\} + 1 = \rk_M(S) + 1\leq |S|+1 = |S \cup s|,
\end{align*}
so it only remains to check that $\rk_M$ is submodular.
Let $S_1, S_2 \subseteq \widetilde E$ and 
  $A_1, A_2 \subseteq E$ such that $\rk_M(S_i) = \rk_{P}(A_i) + 
  |S_i \setminus \pi^{-1}(A_i)|$. Then
\begin{align*}
\rk_M&(S_1) + \rk_M(S_2)
= \rk_{P}(A_1) + |S_1 \setminus \pi^{-1}(A_1)| + \rk_{P}(A_2) + |S_2 \setminus \pi^{-1}(A_2)|  \\
\geq& \rk_{P}(A_1 \cup A_2) + \rk_{P}(A_1 \cap A_2) + |S_1 \setminus \pi^{-1}(A_1)| + |S_2 \setminus \pi^{-1}(A_2)| \quad \text{by submodularity of $\rk_{P}$} \\
\geq& \rk_{P}(A_1 \cup A_2) + \rk_{P}(A_1 \cap A_2) + |(S_1 \cup S_2) \setminus \pi^{-1}(A_1 \cup A_2)| + |(S_1 \cap S_2) \setminus \pi^{-1}(A_1 \cap A_2)| \\
\geq& \rk_M(S_1 \cup S_2)  + \rk_M(S_1 \cap S_2).
\end{align*}
For the second part, let $\gamma \in \Gamma$ and $S \subseteq \widetilde E$.
  For any $A \subseteq E$, we have
   $|S \setminus \pi^{-1}(A)| = |(\gamma \cdot S) \setminus \pi^{-1}(A)|$, and hence 
   $\rk_M(S) = \rk_M(\gamma \cdot S)$. 
\end{proof}
The following lemma implies $\widetilde P$ is a multisymmetric lift of $P$.
\begin{lem}\label{Gfixed}
  Fix notation as in \cref{lift}, and set $M \coloneqq \widetilde{P}$. If $S \subseteq \widetilde E$ is stable under $\Gamma$, then
    $\rk_M(S) = \rk_P(\pi(S)).$
\end{lem}
\begin{proof}
  Since $S$ is stable under the action of $\Gamma$, $S$ is a union of 
  fibers of $\pi$. Hence,
    \begin{align*}
      \rk_{M}(S) =& \min \{ \rk_P(A) + |S \setminus \pi^{-1}(A)| : A \subseteq E \}\quad \text{by the definition of $\rk_{M}$}\\
      =& \min \{ \rk_P(A) + |S \setminus \pi^{-1}(A)| : A \subseteq \pi(S) \} \quad \text{because $\rk_P$ is increasing}\\
      =& \min \bigg\{ \rk_P(A) + \sum_{i\in \pi(S)\setminus 
      A}|\pi^{-1}(i)| : A \subseteq \pi(S) \bigg\}\quad \text{because 
      $S$ is a union of fibers of $\pi$}\\
      = & \min \bigg\{ \rk_P(A) + \sum_{i\in \pi(S)\setminus A}\rk_P(i) 
      : A \subseteq \pi(S) \bigg\} \quad\text{by the definition of $\pi$ 
      in \cref{lift}}\\
      \geq& \rk_P(\pi(S))\quad \text{by submodularity of $\rk_P$.}
    \end{align*}
    On the other hand, taking $A=\pi(S)$ in the definition of $\rk_M(S)$, we have $\rk_{M}(S)\leq \rk_P(\pi(S))$.
\end{proof}
\begin{proof}[Proof of \cref{lift}]
  Let $P$ be a polymatroid, and let $M = \widetilde P$ on $\widetilde E = \widetilde E_1 \sqcup \cdots \sqcup \widetilde E_n$, acted upon by $\Gamma$, be as in the statement of \cref{lift}.
  By \cref{liftwelldefined} and \cref{Gfixed}, $M$ is a multisymmetric lift of $P$. It is minimal because
  \[
    \rk_M(\widetilde E_i) = \rk_P(\pi(\widetilde E_i)) = \rk_P(i) = |\widetilde E_i|.
   \]
   The uniqueness statement follows from \cref{uniqueness} and the fact that the ranks of geometric sets in any lift of $P$ are determined by $P$.
\end{proof}

\subsection{Operations}
The formation of minimal multisymmetric lifts commutes with some polymatroid operations.
Let $F$ be a flat of a polymatroid $P$ on $E$. The \word{restriction} of $P$ to $F$, denoted $P|_F$, is the polymatroid on $F$ with rank function
\[
  \rk_{P|_F}(A) \coloneqq \rk_P(A), \quad A \subseteq F.
\]
There is a lattice isomorphism
\[ \L_{P|_F} \to \{G \in \L_P: G \leq F\},\quad H \mapsto H. \]

If $P_1$ and $P_2$ are polymatroids on $E^1$ and $E^2$, respectively, then their \textit{direct sum} is the polymatroid $P_1 \oplus P_2$ on $E^1 \sqcup E^2$ with rank function
\[
  \rk_{P_1 \oplus P_2}(S) = \rk_{P_1}(S \cap E^1) + \rk_{P_2}(S \cap E^2).
\]
There is a lattice isomorphism
\[
  \L_{P_1} \times \L_{P_2} \to \L_{P_1 \oplus P_2}, \quad (F, G) \mapsto F \sqcup G.
  \]

\begin{lem}\label{lem:sum}
  If $P_1$ and $P_2$ are polymatroids, then
    $\widetilde{P_1 \oplus P_2} = \widetilde P_1 \oplus \widetilde P_2.$
\end{lem}
\begin{proof}
  From Corollary \ref{cor:polymat} and the definition of $P_1 \oplus P_2$, it follows that $\widetilde P_1 \oplus \widetilde P_2$ lifts $P_1 \oplus P_2$.
  If $\Gamma_1$ and $\Gamma_2$ are the groups acting on $\widetilde P_1$ and $\widetilde P_2$, then $\Gamma_1 \times \Gamma_2$ acts on $\widetilde P_1 \oplus \widetilde P_2$, so it is multisymmetric.
  It is minimal by minimality of $\widetilde P_1$ and $\widetilde P_2$, so the lemma holds by the uniqueness statement of \cref{lift}.
\end{proof}

\begin{lem}\label{lem:restriction}
  If $F$ is a flat of a polymatroid $P$, then
  $\widetilde{P|_F} = \widetilde P|_{\pi^{-1}(F)}.$
\end{lem}
\begin{proof}
  The rank functions of both sides are obtained by restriction, so $\widetilde P|_{\pi^{-1}(F)}$ lifts $P|_F$.
  If $\Gamma$ acts on $\widetilde P$, then a subset of $\Gamma$'s factors acts on $\widetilde{P}|_{\pi^{-1}(F)}$, so it is multisymmetric.
  It is minimal because $\widetilde P$ is, so \cref{lift} implies the lemma.
\end{proof}

\section{Building sets and Bergman fans}\label{sec:buildingsets}
Here, we recall the combinatorics of geometric building sets. We then generalize the definition of the Bergman fan of a polymatroid given in the introduction by associating a fan to each polymatroid equipped with a geometric building set.
\subsection{Geometric building sets}
Let $P$ be a polymatroid.
If $\mathscr{G} \subseteq \L_P$ and $F$ is a flat, let
\[ \mathscr{G}_{\leq F} \coloneqq \{ G \in \mathscr{G} : G \leq F\} \]
and write $\max \mathscr{G}$ for the set of maximal elements of $\mathscr{G}$.
\begin{dfn}
  A \word{geometric building set}\footnote{Our terminology is that of \cite[Definition 4.4]{FK04}. In \cite{PP21}, geometric building sets are called ``combinatorial building sets'', but \cite{FK04} uses this term for collections that satisfy only the isomorphism condition. 
  The ``combinatorial building sets'' of \cite{FK04} are the same as the ``building sets'' of \cite{FY04}. Postnikov considers only building sets for the Boolean lattice in \cite[Definition 7.1]{P09}; all such building sets are geometric.} of a polymatroid $P$ is a collection $\G$ of nonempty flats such that for all $F \in \L_P \setminus \{\emptyset\}$, the map
  \begin{equation*}
    \prod_{G \in \max \G_{\leq F}} \L_{P|_G} \to \L_{P|_F} \label{factorization}
    \end{equation*}
  is an isomorphism, and
  \begin{equation*}
 \sum_{G \in \max \G_{\leq F}} \rk(G) = \rk(F). \label{combobuildingset}
\end{equation*}
  If $E \in \G$, then a \word{nested set} of $P$ with respect to $\G$ is a subset $\N \subseteq  \G$ such that for all $\{F_1, \ldots, F_k\} \subseteq \N$ pairwise incomparable with $k \ge 2$, we have that 
  \[ \cl_P(F_1 \cup \cdots \cup F_k)  \not \in \G.\]
\end{dfn}
With respect to a fixed geometric building set, a subset of a nested set is nested, so nested sets form a simplicial complex.
All building sets are henceforth assumed to contain $E$\footnote{In the realizable case, this assumption guarantees the associated wonderful compactification is smooth and can be described as an iterated blow-up \cite[\S4.1]{dCP95}.  Combinatorially, we lose nothing by this assumption \cite[Remark 4.1]{PP21}.}.
\begin{example}
  The maximal geometric building set of $P$ is the collection of all nonempty flats.
  With respect to this building set, the nested sets are flags of nonempty flats.
\end{example}

\begin{lem}\label{liftedbs}
  Let $P$ be a polymatroid, $\widetilde P$ its minimal multisymmetric lift, and $\pi$ as in \cref{lift}.
  If $\mathscr G$ is a geometric building set for $P$, then 
\[ \widetilde{\mathscr G} = \{ \pi^{-1}(G) : G \in \mathscr G\} \cup \{\text{atoms of $\mathscr L_{\widetilde P}$}\} \]
is a geometric building set for $\widetilde P$.
\end{lem}
\begin{proof}
Let $F$ be a flat of $M = \tilde P$.
By \cref{lem:sum} and \cref{lem:restriction}, the map
\[ \prod_{G \in \max \widetilde \G_{\leq F^\geo}} \L_{M|_G} \to \L_{M|_F} \]
  factors into a chain of isomorphisms
  \[
    \prod_{G \in \max \widetilde \G_{\leq F^\geo}} \L_{M|_G}
    \cong \prod_{H \in \max \G_{\leq \pi(F^\geo)}} \L_{\widetilde{P|_H}}
    \cong \L_{\widetilde{P|_{\pi(F^\geo)}}} \cong \L_{M|_{F^\geo}},
\]
and by \cref{Gfixed},
\[
  \sum_{G \in \max \widetilde \G_{\leq F^\geo}} \rk_{M}(G)
  = \sum_{G \in \max \widetilde \G_{\leq F^\geo}} \rk_{P}(\pi(G))
  = \sum_{H \in \max \G_{\leq \pi(F^\geo)}} \rk_{P}(H)
  = \rk_P(\pi(F^\geo)) = \rk_M(F^\geo).
\]
Consequently, by \cref{lem:closure}\cref{lem:closurerank},
\[
  \prod_{G \in \max \widetilde\G_{\leq F}} \L_{M|_G}
   = \prod_{G \in \max \widetilde\G_{\leq F^\geo}}  \L_{M|_G} \times \prod_{i \in F \setminus F^\geo}  \L_{M|_{\{i\}}}
  \cong  \L_{M|_{F^\geo}} \times \prod_{i \in F \setminus F^\geo} \L_{M|_{\{i\}}} \cong \L_{M|_F},
\]
and
\[
  \sum_{G \in \max \widetilde{\mathscr{G}}_{\leq F}} \rk_{M}(G)
  = \rk_M(F^\geo) + \sum_{i \in F \setminus F^\geo} \rk_{M}(i)
  = \rk_M(F^\geo)+|F \setminus F^\geo|
  = \rk_M(F),
\]
as desired.
\end{proof}
\begin{lem}\label{lem:nested}
  If $\G$ is a geometric building set for $P$, then $\N \subseteq \G$ is $\G$-nested if and only if $\widetilde \N = \{ \pi^{-1}(F) : F \in \N\}$ is $\widetilde \G$-nested.
\end{lem}
\begin{proof}
  Set $M = \widetilde P$.
  Let $\{F_1, \ldots, F_k\}\subseteq \mathscr N$ be a set of pairwise incomparable flats.
  By \cref{cor:sublattice}, the geometric flats of $M$ form a sublattice of $\L_M$, isomorphic to $\L_P$ by \cref{lift}. Hence,
  \[
    \cl_M(\pi^{-1}(F_1) \cup \cdots \pi^{-1}(F_k)) \in \widetilde \G  \iff  \cl_P(F_1 \cup \cdots \cup F_k) \in \G.  \qedhere
  \]
\end{proof}

\begin{rmk}[Geometry of building sets]\label{rmk:geomOfBS}
  Let $\{V_i\}_i$ be a subspace arrangement in $V$, defining a polymatroid $P$.
  Let $\G$ be a building set for $P$.
  By \cite[\S1.6]{dCP95}, the \textit{wonderful compactification} of $\{V_i\}_i$ with respect to $\G$ can be constructed by blowing up $\P(V)$ along all subspaces $\cap_{i \in F} \P(V_i)$ with $F \in \G$, first blowing up those of dimension 0, then those of dimension 1, and so on.
Let $\{V_{ij}\}_{ij}$ be a hyperplane arrangement realizing $\widetilde P$, as in \cref{rmk:why-geometric}.
  Blowing up a codimension 1 subvariety is an isomorphism, so \cref{lem:nested} implies that the wonderful compactification of $\{V_i\}_i$ with respect to a building set $\G$ is isomorphic to the wonderful compactification of $\{V_{ij}\}_{ij}$ with respect to $\widetilde \G$.
\end{rmk}

\subsection{Bergman fans}
Let $P$ be a polymatroid on $E$, with minimal lift $\widetilde{P}$ on $\widetilde E$.
Let $\R^{\widetilde E}$ be the vector space spanned by $\e_i$ for $i \in \widetilde E$, and write $\e_S \coloneqq \sum_{i \in S} \e_i$ for $S \subseteq \widetilde E$. 
If $\mathscr S \subseteq 2^{\widetilde{E}}$ is a collection of subsets, write
\[
  \sigma_{\mathscr S} \coloneqq \cone(\e_S : S \in \mathscr S) \subseteq \R^{\widetilde E} / \R (1,1,\ldots,1).
\]
\begin{dfn}\label{dfn:bergmanfan}
  Let $P$ be a polymatroid on $E$, and $\G \subseteq \L_P$ a geometric building set of $P$.
  The \word{Bergman fan} associated to $(P, \G)$ is
  $\Sigma_{P, \G} \coloneqq \{\sigma_{\N}\}_\N$,
  where $\N$ ranges over all $\widetilde\G$-nested sets of $\widetilde P$ such that $\widetilde E \not \in \N$.
\end{dfn}
If $P$ is a matroid and $\G$ is the maximal geometric building set of $P$, then $\widetilde P = P$ and $\widetilde \G = \G$. In this case, $\Sigma_{P,\G}$ coincides with the Bergman fan of \cite[Definition 3.2]{AHK18}.

\begin{lem}
Let $P$ be a polymatroid and $\mathcal{G}$ be the maximal geometric building set of $P$. Then the $\mathcal{G}$-nested sets of $\widetilde{P}$ are in bijection with chains of flats $\mathscr F = \{\emptyset = F_0 \subsetneq F_1 \subsetneq F_2 \subsetneq \cdots \subsetneq F_k \subsetneq {E} \}$ of ${P}$ and a subset $S$ of $\widetilde E$ such that 
\[
\rk_{P}(F \cup \pi(T)) > \rk_{P}(F) + |T| \ \ \text{for every proper flat $F$ in $\mathscr{F}$ and every nonempty subset $T$ of $S\setminus F$.}
\]
\end{lem}

\begin{proof}
Let $\mathcal{N}$ be a nested set of $\tilde{\mathcal{G}}$, which consists of some geometric flats and some atoms. As the join of two geometric flats is geometric, we see that the geometric flats must form a chain $\{\emptyset = F_0 \subsetneq \pi^{-1}(F_1) \subsetneq \pi^{-1}(F_2) \subsetneq \cdots \subsetneq \pi^{-1}(F_k) \subsetneq \widetilde{E} \}$, where $F_1, \dotsc, F_k$ are flats of $P$. Let $S$ be the set of atoms in $\mathcal{N}$. It suffices to check that $S \cup \{\pi^{-1}(F_1), \dotsc, \pi^{-1}(F_k)\}$ is nested if and only if 
\[
\rk_{P}(F \cup \pi(T)) > \rk_{P}(F) + |T| \ \ \text{for every proper flat $F$ in $\mathscr{F}$ and every nonempty subset $T$ of $S\setminus F$.}
\]
If the inequality holds, then
$$\rk_P(F \cup \pi(T)) > \rk_P(F) + |T| \ge \rk_{\widetilde{P}}(\pi^{-1}(F) \cup T)$$
for any $F \in \{F_1, \dotsc, F_k\}$ and all nonempty $T \subset S \setminus \pi^{-1}(F)$. The left-hand side is the rank of the smallest geometric flat containing $\pi^{-1}(F) \cup T$, so the closure of $\pi^{-1}(F) \cup T$ cannot be geometric. This implies that $S \cup \{\pi^{-1}(F_1), \dotsc, \pi^{-1}(F_k)\}$ is nested. 

Now suppose that $S \cup \{\pi^{-1}(F_1), \dotsc, \pi^{-1}(F_k)\}$ is nested, but the inequality fails for some $F \in \{F_1, \dotsc, F_k\}$ and some $T \subset S \setminus \pi^{-1}(F)$. Let $G$ be the smallest geometric flat containing $\pi^{-1}(F) \cup T$. Then
$$\rk_{\widetilde{P}}(\pi^{-1}(F) \cup T) \le  \rk_{\widetilde{P}}(G) = \rk_P(F \cup \pi(T))  \le \rk_P(F) + |T|.$$
If the first inequality is an equality, then $\rk_{\widetilde{P}}(\pi^{-1}(F) \cup T) = \rk_{\tilde{P}}(G)$, and so the closure of $\pi^{-1}(F) \cup T$ is a geometric flat, contradicting that $\mathcal{N}$ is nested. 
Therefore $\rk_{\widetilde{P}}(\pi^{-1}(F) \cup T) < \rk_P(F) + |T| = \rk_{\widetilde{P}}(\pi^{-1}(F)) + |T|$, so there is a circuit $C \subset \pi^{-1}(F) \cup T$ with $C \setminus \pi^{-1}(F)$ nonempty. 
For any $c \in C \setminus \pi^{-1}(F)$, we have $c \in \operatorname{cl}_{\widetilde{P}}(\pi^{-1}(F) \cup C \setminus c)$, and so by Lemma~\ref{lem:closure}(\ref{lem:extraclosure}), 
$$\pi^{-1}(\pi(c)) \subset \operatorname{cl}_{\widetilde{P}}(\pi^{-1}(F) \cup C \setminus c).$$
Using this for all $c \in C \setminus \pi^{-1}(F)$, we have that 
$$\pi^{-1}(\pi(C)) \cup \pi^{-1}(F) \subset \operatorname{cl}_{\widetilde{P}}(\pi^{-1}(F) \cup C), \text{ so } \operatorname{cl}_{\widetilde{P}}(\pi^{-1}(\pi(C)) \cup \pi^{-1}(F)) = \operatorname{cl}_{\widetilde{P}}(\pi^{-1}(F) \cup C).$$
Note that $\operatorname{cl}_{\widetilde{P}}(\pi^{-1}(\pi(C)) \cup \pi^{-1}(F))$ is geometric. Because $C \setminus \pi^{-1}(F) \subset T$, this contradicts that $\mathcal{N}$ is nested. 
\end{proof}

\begin{cor}
  If $P$ is a polymatroid and $\G$ is the maximal geometric building set, then $\Sigma_{P, \G}$ coincides with the Bergman fan $\Sigma_P$ defined in the introduction (\cref{dfn:maximalBergman}). 
\end{cor}

\medskip
\begin{lem}\label{polytope}
  If $P$ is a polymatroid and $\G$ is a geometric building set of $P$, then $\Sigma_{P, \G}$ is a subfan of the normal fan of a convex polytope.
\end{lem}
\begin{proof}
  Let $E$ be the ground set of $P$ and $B$ be the Boolean polymatroid with $\rk_B(i) = \rk_P(i)$ for all $i \in E$.
  Suppose $\G$ is a building set for $P$. The definition of geometric building set implies for any flat $F$ of $P$, $\max \G_{\leq F}$ is a partition of $F$. Hence, $\widetilde \G$ is a building set of both $\widetilde B$ and $\widetilde P$, and there is an inclusion of fans
  $\Sigma_{P, \G} = \Sigma_{\widetilde P, \widetilde \G} \subseteq \Sigma_{\widetilde B, \widetilde \G}$ (In fact, if all single-element subsets of $E$ are flats of $P$, then $\G$ is a geometric building set for $B$, and we obtain $\Sigma_{P,\G} \subset \Sigma_{\widetilde B, \widetilde \G} = \Sigma_{B, \G}$).
  The lattice of flats of $\widetilde B$ is isomorphic to the lattice of subsets of $\widetilde E$,  so $\Sigma_{\widetilde B, \widetilde \G}$ is the normal fan of a convex polytope by \cite[Theorem 7.4]{P09}.
\end{proof}

By \cite{Bri96}, for any polymatroid $P$ and a geometric building set $\G$, $A(\Sigma_{P,\G})$ has the following presentation.
\begin{prop}\label{fanchow}
  The \word{Chow ring} of $\Sigma_{P,\G}$ satisfies 
  \[
    A(\Sigma_{P,\G}) = \Z[z_G : G \in \widetilde \G \setminus \{\widetilde E\}] / I_{P, \G}
  \]
  where $I_{P, \G}$ is the ideal generated by 
  \begin{align*}
      z_{G_1} \cdots z_{G_k} \quad &\text{for any not $\widetilde \G$-nested collection $\{G_1, \ldots, G_k\}$},\\
      \sum_{i \in G} z_G - \sum_{j \in F} z_F \quad &\text{for any } i,j \in \widetilde E.
  \end{align*}
In particular, the Chow ring of $\Sigma_P$ satisfies
\[
  A(\Sigma_P) = \Z[z_F : F \text{ nonempty proper flat of $P$}] \otimes \Z[z_i : i \in \widetilde E] / I_P,
\]
where $I_P$ is an ideal generated by the following polynomials where $z_\emptyset$ is replaced by 1 wherever it appears:
\begin{align*}
  z_{F_1} z_{F_2},  \quad &\text{$F_1$ and $F_2$ are incomparable proper flats of $P$,} \\
  z_F \prod_{i \in T} z_i,  \quad &\text{$F$ is a proper flat and  $T\subseteq \widetilde E\setminus \pi^{-1}(F)$ is nonempty satisfying $\rk(F \cup \pi(T)) \leq \rk(F) + |T|$,} \\
  \sum_{i \in F} z_F - \sum_{j \in G} z_G,  \quad &\text{$i$ and $j$ are elements of $\widetilde E$.}
\end{align*}

\end{prop}

We identify $A^1(\Sigma_{P,\G})$ with the space of piecewise linear functions on the support of $\Sigma_{P, \G}$, modulo global linear functions. Explicitly, a piecewise linear function $\ell$ is a representative of 
\[
  \sum_{\mathbf{u}_F} \ell(\mathbf{u}_F) z_F \in A^1(\Sigma_{P, \G}),
\]
where the sum is over all primitive ray generators $\mathbf{u}_F$ of rays of $\Sigma_{P, \G}$.

\begin{rmk}\label{fypresentation}
  A slightly different presentation of $A(\Sigma_{P, \G})$ is used by \cite{FY04}. 
  For $F \in \widetilde \G$, set
    \[ y_F = \begin{cases} -\sum_{i \in G} z_G, &F = \widetilde E \\ z_F, &\text{otherwise.} \end{cases} \]
    In terms of the $y_F$'s, $A(\Sigma_{P, \G})$ is defined by the ideal $I_{FY}$ generated by
  \begin{align*}
    y_{G_1} \cdots &y_{G_k}, \quad \text{$\{G_1, \ldots, G_k\}$ not $\widetilde \G$-nested and }\\
    \sum_{i \in G}& y_G, \quad i \in \widetilde E.
  \end{align*}
\end{rmk}
\begin{rmk}[Tropicalization and the Bergman fan]\label{tropicalization}
  Suppose $P$ is the matroid realized by an essential hyperplane arrangement $V_1,\ldots,V_n$ in $V$, defined by linear functionals $\ell_1,\ldots,\ell_n$.
  The inclusion
  \[ \P(V \setminus \cup_i V_i) \hookrightarrow \P((\mathbb{F}^*)^n), \quad v \mapsto [\ell_1(v): \ldots: \ell_n(v)] \]
  shows that $\P(V \setminus \cup_i V_i)$ is a very affine variety. Its tropicalization is the support of $\Sigma_P$, by \cite[\S9.3]{S02} and \cite[\S3]{AK06} (see also \cite[Theorem 4.1]{FS05}).
  The corresponding statement for realizable polymatroids does not make sense: the complement of a subspace arrangement may not be very affine, so tropicalization  \cite[Definition 3.2.1]{MS15} is not defined. Nevertheless, if $\{V_i\}_i$ is a subspace arrangement (so $P$ is a polymatroid), then generic hyperplanes $\{V_{ij}\}_{ij}$ realizing $\widetilde P$ as in \cref{rmk:why-geometric} define a subtorus of $\P(\prod_i V/V_i)$. Tropicalizing this torus's intersection with $\P(V \setminus \cup_i V_i)$
  gives the support of $\Sigma_{\widetilde P}$, which coincides with that of $\Sigma_P$.
 
\end{rmk}

\section{Chow rings of polymatroids}\label{sec:chow}
If $\cS$ is a collection of sets, write $\cup \cS$ for the union of the elements of $\cS$.
In \cite[Section 4]{PP21}, Pagaria and Pezzoli define the Chow ring of a polymatroid as follows\footnote{Unlike \cite{PP21}, we use $\Z$-coefficients.}.

\begin{dfn}\label{dfn:ppchow}
 Let $P$ be a polymatroid and $\G \subseteq \L_P$ a geometric building set. 
 The \word{Chow ring} of $(P, \G)$ is
 \[
   \DP(P, \G) \coloneqq \Z[x_F : F \in \G] / I,
 \]
 where $I$ is the ideal generated by
\[
x_{G_1} \cdots x_{G_k} \Big(\sum_{H \geq G} x_H \Big)^b
\]
for $G \in \G$, $\mathscr S = \{G_1, \ldots, G_k\} \subseteq \G$, and
$b \geq \rk_P(G) - \rk_P( \cup \mathscr S_{<G} ).$
\end{dfn}
When $P$ is realizable by an arrangement of subspaces, $\DP(P, \G)$ is the cohomology ring of De Concini \& Procesi's wonderful compactification of the arrangement complement.
In this section, we make use of the theory of Gr\"obner bases. For background on this subject, see \cite[Chapter 15]{E95}.

\begin{thm}\label{thm:isomorphism}
  Let $P$ be a polymatroid on $E$, $\widetilde E$ the ground set of $\widetilde P$, and $\pi\colon  \widetilde E \to E$ the projection.
  There is an isomorphism $\operatorname{DP}(P, \mathscr G) \cong A(\Sigma_{P, \G})$ sending $x_F$ to $y_{\pi^{-1}(F)}$.
\end{thm}
\begin{proof}
    Let $I_{DP} \subseteq \Z[x_F : F \in \G]$ be the defining ideal of 
  $\DP(P, \G)$, and let $I_{FY} \subseteq \Z[y_F : F \in \widetilde \G]$ be the defining ideal of $A(\Sigma_{P,\G})$ as in \cite{FY04} (See
  \cref{fypresentation}). We define the following map on polynomial rings.
  \[
   \ph\colon  \Z[x_F : F \in \G] \to \Z[y_F : F \in \widetilde \G], \quad x_F \mapsto y_{\pi^{-1}(F)}
  \]
    First we show that $\varphi(I_{DP}) \subseteq I_{FY}$. Write 
  $f \in I_{DP}$ for one of the defining relations of $I_{DP}$:
\[ 
  f = \Big(\prod_{F \in \mathscr S} x_F \Big)\Big(\sum_{\mathscr G \ni 
  H \geq G} x_H \Big)^b.
  \]
  By \cite[Theorems 1 and 3]{FY04}, $I_{FY}$ contains the following two types of polynomials:
  \begin{align*}
\prod_{F \in \mathscr S}& y_F, \qquad \mathscr S \text{ not $\widetilde{\mathscr{G}}$-nested,} \\
    \prod_{F \in \N} y_F \Big( \sum_{H \geq G}& y_G \Big)^d, \quad \text{$\N$ a nested antichain, $\cup \N < G$, and $d = \rk(G) - \rk(\cup N)$.} 
\end{align*}
If $\mathscr S$ is not $\mathscr G$-nested, then $\tilde{\mathscr{S}} \coloneqq \{ \pi^{-1}(F) : F \in \mathscr S\}$ is not $\tilde{\mathscr G}$-nested by \cref{lem:nested}. Hence, $\ph(f)$ is divisible by a relation of the first type. 
Otherwise, $\mathscr S$ is $\mathscr{G}$-nested, so $\tilde {\mathscr {S}}$ is $\widetilde{\mathscr{G}}$-nested.
In this case, $\ph(f)$ is divisible by a relation of the second type because
\begin{align*}
 b &\geq
\rk_P(G) - \rk_P(\cup \mathscr S_{<G}) 
  = \rk_{\tilde P}(\pi^{-1}(G)) - \rk_{\tilde P}(\cup \mathscr{\tilde S}_{< \pi^{-1}(G)}) 
  = \rk_{\tilde P}(\pi^{-1}(G)) - \rk_{\tilde P}(\cup \max \mathscr{\tilde S}_{<\pi^{-1}(G)}).
\end{align*}
This proves that $\varphi(I_{DP}) \subseteq I_{FY}$, so $\ph$ descends to
$\bar \ph\colon  \DP(P,\G) \to A(\Sigma_{P,\G})$.

If $F \in \widetilde \G$ is a flat of rank greater than 1, then $y_{F}$ is in the image of $\bar \ph$.
By the linear relation $\sum_{i \in G} y_G = 0$, it follows that $y_i$ is also in the image of $\bar\ph$. Therefore, $\bar\ph$ is surjective. 
It remains to show that $\bar \ph$ is injective.
By \cite[Theorem 2]{FY04}, the generators of $I_{FY}$ in the previous paragraph
are a Gr\"obner basis with respect to any lexicographic monomial order $<$ in which $F_1 \subseteq F_2$ implies $y_{F_1} > y_{F_2}$.
Any such order is an elimination order with respect to $\{y_i: i \in \widetilde E\}$.
By \cite[Proposition 15.29]{E95}\footnote{Eisenbud's proof of this statement works over $\Z$ because all leading coefficients in our Gr\"obner basis are 1.\label{zgbnote}}, 
the generators of $I_{FY}$ in the previous paragraph
that do not involve any $y_i$, $i \in \widetilde E$, are a Gr\"obner basis for $\im(\ph) \cap I_{FY}$.
Any such polynomial is the image of a generator of $I_{DP}$, so $\ph^{-1}(I_{FY}) = I_{DP}$.
This implies $\bar \ph$ is an isomorphism.
\end{proof}
The proof of \cref{thm:isomorphism} also shows the following, originally obtained from \cite[Corollary 2.8]{PP21}.
\begin{cor}\label{cor:gb}
  The defining relations of $\DP(P, \G)$ in 
  \cref{dfn:ppchow} form a Gr\"obner basis with respect to any lexicographic $<$ such that $x_{F_1} < x_{F_2}$ whenever $F_1 \supsetneq F_2$.
\end{cor}
This recovers the monomial basis of \cite{PP21}.
\begin{cor}\label{cor:basis}
  The following monomials are a $\Z$-basis of $\DP(P,\G)$:
  \[
    x_{G_1}^{a_1} \cdots x_{G_k}^{a_k}
  \]
  where $\N = \{G_1, \ldots, G_k\}$ is a nested set of $\G$, and 
  \[ 1 \leq a_i < \rk(F_i) - \rk(\cup \N_{<F_i})\quad \text{for all $1 \leq i \leq k$.} \]
   
\end{cor}
\begin{proof}
Immediate from \cref{thm:isomorphism} and \cite[Corollary 1]{FY04} (or \cref{cor:gb} and \cite[Theorem 15.3]{E95}).
\end{proof}
\begin{rmk}
  In \cite[Corollary 2.8]{PP21}, degree-lexicographic order (also called ``homogeneous lexicographic'' or ``graded lexicographic'' order) is used.
  Since the defining relations of $\DP(P,\G)$ in 
  \cref{dfn:ppchow} are all homogeneous, their initial terms with respect to the lex or degree-lex orders are the the same. 
  Consequently, they are a Gr\"obner basis with respect to one order if and only if they are with respect to the other.
\end{rmk}
\begin{rmk}\label{rmk:geoProof}
      Suppose $P$ is a polymatroid with building set $\G$, realized by an arrangement $\{V_i\}_i$ in $V$. Let $\{V_{ij}\}_{ij}$ realize $\widetilde P$ as in \cref{rmk:why-geometric}. In this case, an alternate proof of \cref{thm:isomorphism} is possible.
  By \cite[\S4]{FY04}, $A(\Sigma_{P, \G})$ is isomorphic to the Chow ring of the wonderful compactification of $\{V_{ij}\}_{ij}$ with respect to $\widetilde \G$.
  By \cref{rmk:geomOfBS}, the wonderful compactifications of $\{V_{ij}\}_{ij}$ with respect to $\widetilde \G$ and $\{V_i\}_i$ with respect to $\G$ are isomorphic. Hence, their Chow rings are isomorphic. The Chow ring of the latter space is isomorphic to $\DP(P, \G)$ by a comparison of the presentations in \cref{dfn:ppchow} and \cite[Theorem 5.2]{dCP95}. 
\end{rmk}
In the remainder of this section, we recover \cite[Theorem 4.7]{PP21} and generalize \cite[Theorem 4.21]{PP21} using \cref{thm:isomorphism} and the tropical Hodge theory of \cite{ADH20}.
If $R$ is a $\Z$-algebra, define $R_\Q \coloneqq R \otimes_\Z \Q$ and $R_\R$ likewise.
\begin{cor}\label{cor:hodge}
  \newcommand{\K}{\mathbb{K}}
  Let $\K = \Q, \R$.
  Let $P$ be a polymatroid of rank $r$ and $\G$ a geometric building set.
  Let $\ell$ be any $\K$-valued strictly convex piecewise linear function on $\Sigma_{P,\G}$, viewed as an element of $A^1(\Sigma_{P,\G})_\K \cong \DP^1(P,\G)_\K$.
  \begin{enumerate}\itemsep 5pt
  \item \label{poincare} (Poincar\'e duality) There is an isomorphism 
    \[ \deg\colon  \DP(P,\G)^{r-1} \to \Z, \]
    and for all $0 \leq k < r/2$, the pairing
      \[ \DP^k(P,\G) \times \DP^{r-k-1}(P, \G) \to \Z, \quad (a,b) \mapsto \deg(ab)  \]
    is non-degenerate.
\item \label{lefschetz} (Hard Lefschetz) For every $0 \leq k < r/2$, the multiplication map
  \[ \DP^k(P,\G)_\K \to \DP^{r-k-1}(P,\G)_\K, \quad a \mapsto \ell^{r - 2k-1} a \]
  is an isomorphism.
\item \label{hodge-riemann} (Hodge-Riemann) For every $0 \leq k < r/2$, the bilinear form
  \[
    \DP^k(P,\G)_\K \times \DP^k(P,\G)_\K \to \K, \quad (a,b) \mapsto (-1)^k \deg(\ell^{r-2k-1} ab)
  \]
  is positive definite on the kernel of multiplication by $\ell^{r-2k}$.
  \end{enumerate}
\end{cor}
\begin{proof}
  Let $M = \widetilde P$.
  By \cite[Theorem 4]{FY04}, $\Sigma_M$ is a refinement of $\Sigma_{P,\G}$. In particular, the two fans have the same support.
  By \cite[Proposition 2.4]{AHK18} and \cref{polytope}, both fans are subfans of the normal fans of convex polytopes, so both fans support strictly convex piecewise linear functions. By \cite[Proposition 2]{FY04}, $\Sigma_{P, \mathscr{G}}$ is a smooth fan. 

  The desired statements now follow by applying \cite[Theorem 1.6]{ADH20}, \cite[Proposition 5.2]{AHK18}, and \cite[Theorems 6.19 and 8.8]{AHK18}.
  Loosely, \cite[Theorem 1.6]{ADH20} says properties \cref{poincare}, \cref{lefschetz}, and \cref{hodge-riemann} hold for $A(\Sigma_{P,\G})_\R \cong \DP(P,\G)_\R$ if and only if they hold for $A(\Sigma_M)_\R$, and \cite{AHK18} verifies them for $A(\Sigma_M)_\R$.

  Properties \cref{lefschetz} and \cref{hodge-riemann} for $\DP(P, \G)_\Q$ follow immediately from those for $\DP(P, \G)_\R$.
  For \cref{poincare}, note that \cite{ADH20}'s Poincar\'e duality arguments go through over $\Z$. (Explicitly, one must check statements 6.6--6.9, and Propositions 6.16, 6.17 of \cite{ADH20}.)
\end{proof}
\begin{rmk}\label{ConeRemark}
  In \cite{PP21}, \cref{cor:hodge} is proved for $\ell$ in the \word{$\sigma$-cone} \cite[Definition 4.15]{PP21}, the positive span of 
  \[
    - \sum_{G \in \G_{\geq F}} x_G, \quad F \in \G
  \]
  in $\DP(P, \G)$. The $\sigma$-cone is generally a proper subset of the cone of strictly convex piecewise linear functions on $\Sigma_{P, \G}$.
  For example, if $M$ is a loopless matroid on $E$ and $\G$ is its maximal building set, then for any $i \in E$,
  \[
    \beta \coloneqq \sum_{i \not \in F} z_F = - \sum_{|G| > 1} \big(|G| - 1\big) y_G
  \]
  is in the closure of the cone of strictly convex piecewise linear functions on $\Sigma_M$ \cite[Proposition 4.3, Lemma 9.7]{AHK18}.
  However, $\beta$ may not be in the closure of the $\sigma$-cone, e.g. when $M$ is Boolean of rank at least 3.
  For comparison of the $\sigma$-cone to the ample cone of the wonderful compactification, see \cite[Remark~4.22]{PP21}.
\end{rmk}
\appendix
\section{Combinatorics of the Bergman fan of Boolean polymatroids}
\label{appendix:polytope}

In this appendix we describe the combinatorics of the Bergman fans of 
Boolean polymatroids, proving in particular that they are the normal 
fans of polypermutohedra as stated in the introduction. We also give a description of polypermutohedra as a Minkowski sum of simplices. 
Throughout this appendix, we let $\pi\colon \widetilde{E} \to E$ be a surjective map of 
finite sets, with associated Boolean polymatroid $B(\pi)$ on $E$ given by 
the rank function $\rk_{B(\pi)}(A) = |\pi^{-1}(A)|$ for $A \subseteq E$.
We write $n$ for the cardinality of $E$.

\subsection{The Bergman fan as a configuration space}

\begin{dfn}
  Let $\mathbf{w} = (w_i)_i \in \R^{\widetilde{E}}$ be a weight on the elements of 
  $\widetilde{E}$. Write $\Lowest_\pi(\mathbf{w})$ for the set of $i \in 
  \widetilde{E}$ such that $i$ has minimal weight among the 
  elements of $\pi^{-1}(\pi(i))$ with respect to 
  $\mathbf{w}$. We equip 
  $\Lowest_\pi(\mathbf{w})$ with the natural partial order given by $i \preceq j$ 
  if $w_i \leq w_j$. 
\end{dfn}
Adding a multiple of the all ones vector to $\mathbf{w}$ 
  does not change $\Lowest_\pi(\mathbf{w})$, so $\Lowest_\pi(\mathbf{w})$ is well defined 
  for $\mathbf{w} \in \R^{\widetilde{E}}/\R(1,1,\ldots,1)$.

\begin{lem}\label{conf-space}
  Two points $\mathbf{w},\mathbf{v} \in 
  \mathbb{R}^{\tilde{E}}/(1,1,\ldots,1)$ lie in the relative interior of 
  the same cone of the $\Sigma_{B(\pi)}$ if and only if 
  $\Lowest_\pi(\mathbf v) = \Lowest_\pi(\mathbf w)$ as posets.
\end{lem}
\begin{proof}
  Recall that a cone $\sigma_{\mathscr F, 
  S}$ of $\Sigma_{B(\pi)}$ is specified by a chain of sets $\mathscr F = 
  \{\emptyset \subsetneq F_1 \subsetneq \cdots \subsetneq F_k \subsetneq E\}$ 
  and a set $S \subseteq \widetilde E$ such that $S$ 
  does not contain a fiber of $\pi$.
  The relative interior of $\sigma_{\mathscr F, S}$ contains $\mathbf w$ 
  if and only if the underlying set of $\Lowest_\pi(\mathbf w)$ 
  is equal to $\widetilde E 
  \setminus S$ and $i \prec j$ whenever there exists $r$ 
  such that $i \not\in F_r$ and $j \in F_r$.
Therefore $\Lowest_{P}(\mathbf w)$ can be recovered from $\sigma_{\mathscr{F}, A}$ and vice versa.
\end{proof}

\subsection{The Bergman fan as the normal fan of a polytope}
Recall that an \emph{ordered transversal} of $\pi$ is a sequence 
$s_1, \ldots, s_n$ of elements of $\widetilde{E}$ such that each 
fiber of $\pi$ contains exactly one element of the sequence.
\begin{dfn}
Given real numbers $0 \leq c_1 < c_2 < \ldots < c_n$, define the 
  associated polypermutohedron $Q(\pi; c_1, c_2, \ldots, c_n)$ as the convex hull of the vectors $
\mathbf{v}_{s_1,s_2, \ldots, s_n} \coloneqq c_1 \e_{s_1} + c_2 \e_{s_2} + \ldots 
  + c_n \e_{s_n},
$
where $s_1,s_2, \ldots, s_n$ runs over all ordered transversals of 
  $\pi$ and $\e_{i} \in \R^{\widetilde{E}}$ is the standard basis vector 
  of $i \in \widetilde{E}$.
\end{dfn}

\begin{lem}\label{minima}
  Let $\mathbf{w} = (w_i)_i \in \R^{\widetilde{E}}$ be a weight on the elements of 
  $\widetilde{E}$, and let
  $s_1,s_2,\ldots,s_n$ be an ordered transversal of $\pi$. Denote by 
  $\langle -,- \rangle$ the standard dot product on $\R^{\widetilde{E}}$. 
  Then the linear functional $\langle \mathbf{w}, - \rangle$ achieves its 
  minimum over $Q(\pi; c_1, c_2, \ldots, c_n)$ at the vector $\mathbf 
  v_{s_1, s_2, \ldots, s_n}$ if and only if
\begin{enumerate}\itemsep 5pt
    \item $s_j$ has minimum weight among the elements of 
      $\pi(\pi^{-1}(s_j))$ with respect to $\mathbf{w}$ for all $j$, and 
      \label{it:fastest}
    \item $w_{s_1} \leq w_{s_2} \leq \ldots \leq w_{s_n}$. \label{it:incr}
\end{enumerate}
\end{lem}
\begin{proof}
  Suppose that $\mathbf v_{s_1, s_2, \ldots, s_n}$ minimizes $\langle 
  \mathbf w, - \rangle$ over $Q(\pi; c_1, c_2, \ldots, c_n)$. Assume contrary to \cref{it:fastest} that there 
  is some $j$ and $r \in \pi(\pi^{-1}(s_j))$ such that $w_{s_j} > 
  w_{r}$. Then replacing $s_j$ with $r$ gives another ordered 
  transversal of $\pi$ whose corresponding vector has smaller dot 
  product with $\mathbf w$:
  \[
    \langle \mathbf{w}, v_{s_1,\ldots, s_j, \ldots, s_n} \rangle = c_1 
    w_{s_1} + \ldots + c_j w_{s_j} + \ldots + c_n w_{s_n} > c_1 
    w_{s_1} + \ldots + c_j w_{r} + \ldots + c_n w_{s_n} = \langle 
    \mathbf{w}, v_{s_1,\ldots, r, \ldots, s_n} \rangle.
  \]
  This proves \cref{it:fastest}. Now assume contrary to \cref{it:incr} that there is some $j$ such that
  $w_{s_j} > w_{s_{j+1}}$. Then switching the order of $s_j$ and 
  $s_{j+1}$ gives another ordered transversal, and thus another vector 
  $\mathbf v' \in Q(\pi; c_1,c_2,\ldots, c_n)$, where by assumption 
  $\langle \mathbf{w}, \mathbf{v}_{s_1,s_2, \ldots, s_n} \rangle \leq 
  \langle \mathbf{w}, \mathbf{v}' \rangle$. Thus $c_j w_{s_j} + 
  c_{j+1}w_{s_{j+1}} \leq c_j w_{s_{j+1}} + c_{j+1}w_{s_{j}}$,
  which contradicts the fact that $ab + cd > ac + bd$ whenever $a > c$ and $b > d$. This proves \cref{it:incr}.

  For the other direction, assume that $\mathbf{w}$ satisfies the two 
  conditions. If $\langle \mathbf{w}, - \rangle$ achieves its minimum 
  over $Q(\pi; c_1, c_2, \ldots, c_n)$ on a vector $\mathbf 
  v_{s_1',s_2',\ldots,s_n'}$, then by the first direction we must have 
  that $w_{s_j} = w_{s_j'}$. Therefore $\langle \mathbf{w},-\rangle$ also achieves it 
  minimum on $\mathbf v_{s_1,s_2,\ldots, s_n}$.
\end{proof}

\begin{prop}\label{prop:normal}
  The inner normal fan of $Q(\pi;c_1,c_2,\ldots,c_n)$, modulo the all ones vector, is $\Sigma_{B(\pi)}$.
\end{prop}
\begin{proof}
  By Lemma~\ref{minima}, the vertices of $Q(\pi; c_1,c_2,\ldots,c_n)$ correspond to ordered transversals of $\pi$, and the set of vertices on which a given linear functional $\langle \mathbf{w},-\rangle$ achieves its minimum is equivalent to the data of $\Lowest_{P}(\mathbf{w})$.
Therefore the proposition follows by Lemma~\ref{conf-space}.
\end{proof}

\subsection{Minkowski sums of simplices}

For $S \subseteq \tilde{E}$, let $\Delta_{S}$ be the convex hull of the vectors $\e_i$, for $i \in S$. 

\begin{prop}
The polytope $Q(\pi; 1, 2, \dotsc, n)$ is the Minkowski sum $\sum_{\{i, j\} \subseteq E} \Delta_{\pi^{-1}(\{i, j\})}$. 
\end{prop}

In the sum, we allow $i = j$. When $\pi$ is a bijection, this recovers the description of the usual permutohedron as the graphical zonotope of the complete graph. 

\begin{proof}
The proof of Proposition~\ref{prop:normal} shows that the inner normal fan of $\Delta_{\pi^{-1}(S)}$ is a coarsening of $\Sigma_{B(\pi)}$ for any $S \subseteq E$. In particular, the inner normal fan of the Minkowski sum $\sum_{\{i, j\} \subseteq E} \Delta_{\pi^{-1}(\{i, j\})}$ is a coarsening of $\Sigma_{B(\pi)}$. We may then find all vertices of the Minkowski sum by choosing a maximal cone of $\Sigma_{B(\pi)}$ and finding the vertex of the Minkowski sum on which any vector in the interior of this cone achieves its minimum. 

The maximal cones of $\Sigma_{B(\pi)}$ correspond to maximal chains $\mathcal{F} = \{\emptyset \subsetneq F_1 \subsetneq \dotsb \subsetneq F_k \subsetneq E\}$ of subsets of $E$ and subsets $S \subseteq \tilde{E}$ such that $|\pi^{-1}(i) \setminus S| = 1$ for all $i$. This data is equivalent to the data of an ordered transversal $s_1, \dotsc, s_n$ of $\pi$. 
Choose a maximal cone of $\Sigma_{B(\pi)}$ corresponding to an ordered transversal $s_1, \dotsc, s_n$, and choose a vector in the relative interior of this cone. 
We can compute the vertex of the Minkowski sum on which this vector achieves its minimum by adding up the minimal vertices of each summand. 
The minimal vertex of a summand of the form $\Delta_{\pi^{-1}(i)}$ is $\e_{s_k}$, where $k$ is the unique element of $\tilde{E}$ such that $\pi(s_k) = i$. The minimal vertex of a summand of the form $\Delta_{\pi^{-1}(\{i, j\})}$ for $i \not=j$ is $\e_{s_{\ell}}$, where $\ell$ is the smaller index of the two elements of $\pi^{-1}(\{i, j\}) \cap \{s_1, \dotsc, s_n\}$. 
We see that the minimal vertex of the Minkowski sum is $\sum_{i=1}^{n} i \e_{s_i}$, as desired.  
\end{proof}

\begin{remark}
One can deduce from the theory of building sets, e.g., \cite[Proposition 7.5]{P09}, that $\Sigma_{B(\pi)}$ is the normal fan of the Minkowski sum $\sum_{\emptyset \not= S \subseteq E} \Delta_{\pi^{-1}(S)}$. 
\end{remark}

\raggedbottom
\bibliography{arxivv2}
\bibliographystyle{amsalpha}

\end{document}